\documentclass[9pt, reqno]{amsart}

\usepackage{xcolor}

\newtheorem{theorem}{Theorem}[section]
\newtheorem{lemma}[theorem]{Lemma}
\newtheorem{proposition}[theorem]{Proposition}
\newtheorem{corollary}[theorem]{Corollary}

\theoremstyle{definition}
\newtheorem{definition}[theorem]{Definition}

\theoremstyle{remark}
\newtheorem{remark}[theorem]{Remark}

\numberwithin{equation}{section}

\begin{document}

\title[
linear elliptic equations with $L^2$-drifts of negative divergence]{
On the contraction properties for weak solutions to linear elliptic equations with $L^2$-drifts of negative divergence}

\author{Haesung Lee }
\address{Department of Mathematics and Big Data Science, Kumoh National Institute of Technology, Gumi, Gyeongsangbuk-do 39177, Republic of Korea}
\email{fthslt@kumoh.ac.kr, \,  fthslt14@gmail.com}







\begin{abstract}
We show the existence and uniqueness as well as boundedness of weak solutions to linear elliptic equations with $L^2$-drifts of negative divergence and singular zero-order terms which are positive. Our main target is to show the $L^r$-contraction properties of the unique weak solutions. Indeed, using the Dirichlet form theory, we construct a sub-Markovian $C_0$-resolvent of contractions and identify it to the weak solutions.  Furthermore, we derive an $L^1$-stability result through an extended version of the $L^1$-contraction property.
\end{abstract}
\maketitle

\section{Introduction}
\noindent This paper is devoted to studying the following Dirichlet boundary value problem for linear elliptic equations in divergence form on a bounded open subset $U$ of $\mathbb{R}^d$:
\begin{equation} \label{undeq}
\left\{
\begin{alignedat}{2}
-\text{div}(A \nabla u) + \langle B, \nabla u \rangle + (c +\alpha) u &=f -\text{div}F&& \quad \mbox{in $U$}\\
u &= 0 &&\quad \mbox{on $\partial U$},
\end{alignedat} \right.
\end{equation}
where $A$ is a $d \times d$ matrix of functions satisying \eqref{assump}, $B$ and $F$ are vector fields on $U$, $c$  and $f$ are real-valued functions on $U$ and $\alpha \geq 0$ is a constant. Here $u$ is called a weak solution to \eqref{undeq} if $u \in H^{1,2}_0(U)$ satisfies for any $\varphi \in C_0^{\infty}(U)$,
\begin{equation} \label{(maineq)}
\int_{U} \langle A \nabla u, \nabla \varphi \rangle+ \big( \langle B,  \nabla u \rangle +(c+\alpha) u \big)\varphi dx   = \int_{U} f \varphi  +\langle F, \nabla \varphi \rangle dx,
\end{equation}
where both integral terms are well defined in $\mathbb{R}$. 
We  also consider the dual problem of \eqref{undeq},
\begin{equation} \label{undeqdual}
\left\{
\begin{alignedat}{2}
-\text{div}(A \nabla w+wB)  + (c +\alpha) w &=f -\text{div}F&& \quad \mbox{in $U$}\\
w &= 0 &&\quad \mbox{on $\partial U$}.
\end{alignedat} \right.
\end{equation}
Similarly to the above, $w$ is called a weak solution to \eqref{undeqdual} if $w \in H^{1,2}_0(U)$ satisfies  for any $\varphi \in C_0^{\infty}(U)$,
\begin{equation} \label{(maineqdual)}
\int_{U} \langle A \nabla w + w B,  \nabla \varphi \rangle+ (c+\alpha) w \varphi dx   = \int_{U} f \varphi  +\langle F, \nabla \varphi \rangle dx,
\end{equation}
where both integral terms are well defined in $\mathbb{R}$. 
\\
To the author's knowledge, \cite{S65} is the first article to systematically derive the existence and uniqueness of weak solutions to \eqref{undeq}, where $\alpha \geq \bar{\lambda}$ for some constant $\bar{\lambda}>0$, $f=0$, $F \in L^{2}(U)$, $B \in L^d(U, \mathbb{R}^d)$ and $c \in L^{\frac{d}{2}}$ if $d \geq 3$, and $B \in L^{2q}(U, \mathbb{R}^d)$ and $c \in L^q(U)$ for some $q>1$ if $d=2$. 
Indeed, the regularity condition on the coefficients above is minimal in the sense that if $u$ is a weak solution to \eqref{undeq}, then \eqref{(maineq)} holds for all $\varphi \in H^{1,2}_0(U)$, so that the corresponding bilinear form may satisfy the coercivity. In \cite{T73}, some degenerate coefficients $A$ were allowed, and the restriction of $\alpha \geq \overline{\lambda}$ in the situation of \cite{S65} can be relaxed to an arbitrary $\alpha \geq 0$ in case of $c \geq 0$.
More singular drifts such as $B \in L^2(U, \mathbb{R}^d)$ are covered in \cite[Section 2.2.3]{K07} where the existence and uniqueness of weak solutions to \eqref{undeq} are shown under the assumption that $A= id$, $c+\alpha=0$ and ${\rm div} B =0$ weakly in $U$.  Later, assuming that $d \geq 3$, $A=id$, $B \in L^2(U, \mathbb{R}^d)$ with ${\rm div} B \leq 0$ weakly in $U$, $c \in L^{\frac{2d}{d+2}}(U)$ with $c \geq 0$ in $U$ and $\alpha=0$, the existence, uniqueness and boundedness of weak solutions to \eqref{undeq} are shown in \cite[Section 4]{KT20}. 
We refer to the recent articles \cite{KK19, KT20, K21, K22} concerning the existence and uniqueness of weak solutions in $W_0^{1,p}(U)$ ($p>2$) of \eqref{undeq} where singular drifts $B$ are covered. \\
The main target of this paper is to show that by using Dirichlet form theory, the unique weak solutions to \eqref{undeq} with singular coefficients (under assumption {\bf (A)} below) actually have $L^r$-contraction properties with $r \in [1, \infty]$.
 More precisely, by constructing a sub-Markovian $C_0$-resolvent of contractions $(\overline{G}_{\alpha})_{\alpha>0}$ on $L^1(U)$ associated with an $L^1$-closed operator $(\overline{L}, D(\overline{L}))$, we aim to inherit nice properties of $(\overline{G}_{\alpha})_{\alpha>0}$ such as $L^r$-contraction properties with $r \in [1, \infty]$ to the unique solution to \eqref{undeq}. We use the key idea from \cite{S99} (cf. \cite[Chapter 2]{LST22}) and partially extended it to the case where the components of $A$ have no weak differentiability, ${\rm div} B \leq 0$ weakly in $U$ and $c \in L^{2_*}(U)$ with $c \geq 0$, where $2_*:= \frac{2d}{d+2}$ if $d \geq 3$ and $2_* \in (1,2)$ is arbitrary but fixed if $d=2$. The main difference in this paper from \cite{S99} is that our Dirichlet form $(\mathcal{E}^0, D(\mathcal{E}^0))$ is a non-symmetric and sectorial Dirichlet form, and hence the approximation via the corresponding resolvent $(G^0_{\alpha})_{\alpha>0}$ is mainly used rather than the corresponding semigroup $(T^0_t)_{t>0}$ (see Remark \ref{dirichdiffer}).
Before presenting our first main result, let us consider our basic assumption: \\ \\
{\it
{\bf (A)}\; $U$ is a bounded open subset of $\mathbb{R}^d$ with $d\geq 2$. $A=(a_{ij})_{1 \leq i,j \leq d}$ is a $d\times d$  (possibly non-symmetric) matrix of functions in $L^{\infty}(U)$ such that for some constants $\lambda, \Lambda>0$
\begin{equation} \label{assump}
\langle A(x) \xi, \xi \rangle \geq \lambda \|\xi \|^2,   \quad \max_{1 \leq i,j \leq d} |a_{ij}(x)| \leq  \Lambda \quad \text{ for all $x \in U$ and $\xi \in \mathbb{R}^d$}.
\end{equation}
$B \in L^2(U, \mathbb{R}^d)$ satisfies that ${\rm div} B \leq 0$ weakly in $U$, i.e.
\begin{equation} \label{weakdivne}
\int_{U} \langle B, \nabla \varphi \rangle dx \geq 0\;\; \; \text{ for all $\varphi \in C_0^{\infty}(U)$ with $\varphi \geq 0$}.
\end{equation}
$\alpha \in [0, \infty)$ and $c \in L^{2_*}(U)$ with $c \geq 0$, where $2_*:= \frac{2d}{d+2}$ if $d \geq 3$ and $2_* \in (1,2)$ is arbitrary but fixed if $d=2$.
}
\begin{theorem} \label{theomain}
Assume {\bf (A)}. Then, the following hold:
\begin{itemize}
\item[(i)] {\bf (Existence, uniqueness and boundedness).}
Let $f \in L^{2_*}(U)$ and $F \in L^2(U, \mathbb{R}^d)$. Then, there exist a unique weak solution $u$ to \eqref{undeq} and a constant $C_1>0$ which only depends on $\lambda$, $d$ and $|U|$ such that
\begin{equation} \label{enerestim}
\|u\|_{H_0^{1,2}(U)} \leq C_1 \left( \|f\|_{L^{2_*}(U)} + \| F\|_{L^2(U)} \right). 
\end{equation}
Moreover, if $q \in (d/2, \infty)$ with $2_* \leq q$, $f \in L^{q}(U)$ and $F \in L^{2q}(U)$, then $u \in L^{\infty}(U)$ and
\begin{equation} \label{linfinestim}
\|u\|_{L^{\infty}(U)} \leq C_2 \left(\|f\|_{L^{q}(U)} + \| F\|_{L^{2q}(U)}  \right),
\end{equation}
where $C_2>0$ is a constant which only depends on $\lambda$, $d$, $q$ and $|U|$.
\item[(ii)] {\bf (Contraction properties).}
Let
$\alpha \in (0, \infty)$, $r \in [1, \infty]$, $q \in (d/2, \infty)$, $f \in L^r(U) \cap L^{2_*}(U)$ and $F \in L^{2q}(U,\mathbb{R}^d)$. Then,
$u \in H^{1,2}_0(U)$ as in (i) satisfies $u \in L^r(U)$ and 
\begin{align} \label{lrestim}
\| u \|_{L^r(U)} \leq \alpha^{-1} \| f\|_{L^r(U)} +C_2|U|^{1/r}\|F\|_{L^{2q}(U)},
\end{align}
where $C_2>0$ is the constant as in (i). 
\end{itemize}
\end{theorem}
\noindent
We mention that the main idea for the proof of Theorem \ref{theomain}(i) has already been established by the proofs in \cite[Section 2.2.3]{K07} and \cite[Section 4]{KT20}, where the existence of a bounded weak solution $w$ to \eqref{undeqdual} is crucially used to show the uniqueness of the weak solution to \eqref{undeq}. But, Theorem \ref{theomain}(ii) which shows the $L^r$-contraction properties of weak solutions is our main achievement, and its proof is based on a nontrivial application of the Dirichlet form theory and the uniqueness of weak solutions to \eqref{undeq}. 
It is remarkable that under the assumption {\bf (A)} with $f \in L^{2_*}(U)$ and $F=0$, the sequence of unique weak solutions to \eqref{undeq} converges to $0$ as $\alpha$ goes to infinity.
\\
Our second main result is concerned with $L^1$-stability for the unique solutions to \eqref{undeq} which is indeed a corollary of an extended version of the $L^1$-contraction property (Theorem \ref{theo4.2}).
As in Theorem \ref{theomain}(ii), the following result is non-trivial even in the case where $A=id$, $B=0$ and $c \in L^{2_*}(U)$ with $c \geq 0$ on $U$.
\begin{corollary}[\bf $L^1$-stability] \label{stabil}
Assume {\bf (A)}. Let $\alpha \in (0, \infty)$, $f \in L^{1}(U)$ and $F \in L^2(U, \mathbb{R}^d)$.
For each $n \geq 1$, let $A_n=(a^n_{ij})_{1 \leq i,j \leq d}$ be a matrix of functions in $L^{\infty}(U)$ such that $\langle A_n(x) \xi , \xi \rangle \geq \lambda \| \xi\|^2$ for any $x \in U$ and $\xi \in \mathbb{R}^d$, where $\lambda>0$ is the constant in \eqref{assump}. For each $n \geq 1$, let $B_n  \in L^2(U, \mathbb{R}^d)$, $c_n \in L^{2_*}(U)$, $f_n \in L^1(U)$ and $F_n \in L^2(U, \mathbb{R}^d)$, and let $u_n$ be a weak solution to 
\begin{equation} \label{undeqappro}
\left\{
\begin{alignedat}{2}
-\text{\rm div}(A_n \nabla u_n) + \langle B_n, \nabla u_n \rangle + (c_n +\alpha) u_n &=f_n -\text{\rm div}F_n&& \quad \mbox{in $U$}\\
u_n &= 0 &&\quad \mbox{on $\partial U$}.
\end{alignedat} \right.
\end{equation}
Let $u$ be a weak solution to \eqref{undeq}. Then, the following estimate holds:
\begin{align}
&\|u_n-u\|_{L^1(U)} \leq  \alpha^{-1} C_4  \Big( \|B-B_n\|_{L^2(U)}  + N\|c-c_n\|_{L^{2_*}(U)} \Big)  \nonumber \\
&\quad + C_3 \|(A_n-A) \nabla u \|_{L^2(U)} + \alpha^{-1} \|f-f_n\|_{L^1(U)} + C_3\|F-F_n\|_{L^2(U)}, \label{stabilest}
\end{align}
where $C_4= C_1 \left( \|f\|_{L^{2_*}(U)} + \| F\|_{L^2(U)} \right)$, $C_3= |U|^{1/2}C_1$ and
$N=\frac{2(d-1)}{d-2}$ if $d \geq 3$ and $N =\frac{2^*}{2} |U|^{\frac{1}{2^*}}$ with $2^*:=\frac{2_*}{2_*-1}$ if $d=2$ and
$C_1>0$ is the constant as in Theorem \ref{theomain}(i). 
In particular, if for some $M>0$, $\max_{1 \leq i,j \leq d}\|a^n_{ij}\|_{L^{\infty}(U)} \leq M$  for all $n \geq 1$ and $\lim_{n \rightarrow \infty}a^n_{ij}=a_{ij}$ a.e. for each $1 \leq i,j \leq d$, $\lim_{n \rightarrow \infty} B_n = B$ and $\lim_{n \rightarrow \infty} F_n = F$ in $L^2(U, \mathbb{R}^d)$,
$\lim_{n \rightarrow \infty} c_n = c$ in $L^{2_{*}}(U)$ and $\lim_{n \rightarrow \infty} f_n =f$ in $L^1(U)$, then $\lim_{n \rightarrow \infty} u_n =u$ in $L^1(U)$.
\end{corollary}
\noindent
The proofs of Theorem \ref{theomain} and Corollary \ref{stabil} are described in Section \ref{sec5}.
We expect our $L^1$-stability result to be meaningful numerically. Of course, the weak compactness of $H^{1,2}_0(U)$ and \eqref{enerestim} in Theorem \ref{theomain}(i) make us construct a sequence of functions $(\widetilde{u}_n)_{n \geq 1}$ in $H^{1,2}_0(U)$ satisfying that $\lim_{n \rightarrow \infty} \widetilde{u}_n =u$ weakly in $H^{1,2}_0(U)$ where $\widetilde{u}_n$ is a weak solution to an approximation of \eqref{undeq}. However, $(\widetilde{u}_n)_{n \geq 1}$ is
 eventually extracted as a subsequence of $(u_n)_{n \geq 1}$ in Corollary \ref{stabil}. On the other hand, our $L^1$-stability result allows us to explicitly construct a sequence of functions $(u_n)_{n \geq 1}$ that strongly converges to $u$ in $L^1(U)$, where $u_n$ is a solution to the approximation of \eqref{undeq}. 
We refer to \cite[Theorems 2.8, 6.4]{Kr21} for a relevant stability result. \\
\noindent
Our paper is structured as follows. Section \ref{sec2} presents our notations mainly covered in this paper and introduces basic knowledge of Dirichlet form theory. In Section \ref{sec3}, a sub-Markovian $C_0$-resolvent of contractions on $L^1(U)$ associated with an $L^1$-closed operator is constructed based on the idea of \cite{S99}. We construct in Section \ref{sec4} bounded weak solutions to \eqref{undeq} and \eqref{undeqdual} using the methods of weak convergence and Moser's iteration. Finally, in Section \ref{sec5} the uniqueness of weak solutions to \eqref{undeq} and our $L^1$-stability are proved.

\section{Framework} \label{sec2}
\noindent For notations dealt with in this paper, we refer to \cite[Notations and Conventions]{LST22}.
Throughout this paper, we consider the Euclidean space $\mathbb{R}^d$, equipped with the Euclidean inner product $\langle \cdot , \cdot \rangle$ and the Euclidean norm $\| \cdot \|$. The indicator function of $A \subset \mathbb{R}^d$ is denoted by $1_A$. For $a,b \in \mathbb{R}$ with $a \leq b$, we write $a \wedge b=b \wedge a :=a$. Denote the Lebesgue measure by $dx$ and write $|E|:=dx(E)$, where $E$ is a Lebesgue measurable set. 
Let $U$ be an open subset of $\mathbb{R}^d$. Denote the space of all continuous functions on $U$ by $C(U)$.  Write $C_0(U):= C(U)_0$. For $k \in \mathbb{N} \cup \{ \infty \}$, denote by $C^k(U)$ the set of all $k$-times differentiable functions and define $C^k_0(U):=C^k(U) \cap C_0(U)$. Define $C_0^{\infty}(U, \mathbb{R}^d):=\{F=(f_1, \ldots, f_d) : f_i \in C^{\infty}_0(U), i=1,2, \ldots, d  \}$.
For each $p \in [1, \infty]$, the usual $L^p$-space on $U$ with respect to $dx$ is denoted by $L^p(U)$ equipped with $L^p$-norm $\| \cdot \|_{L^p(U)}$. Denote the space of $L^p$-vector fields on $U$ by $L^p(U, \mathbb{R}^d)$ equipped with the norm $\|F\|_{L^p(U)} := \big \|  \|F\| \big\|_{L^p(U)}$, $F \in L^p(U, \mathbb{R}^d)$. 
The Sobolev space $H^{1,2}(U)$ is defined to be the space of all functions $f \in L^2(U)$ for which $\partial_j f \in L^2(U)$, $j=1, \ldots, d$ equipped with the norm
$\|f\|_{H^{1,2}(U)}:=\left(\|f\|^2_{L^2(U)}  +\sum_{i=1}^d \|\partial_i f \|^2_{L^2(U)} \right)^{1/2}$, where $\partial_i f$ is is the $i$-th weak partial derivative of $f$ on $U$.
$H^{1,2}_0(U)$ denotes the closure of $C_0^{\infty}(U)$ in $H^{1,2}(U)$ equipped with the norm $\|\cdot \|_{H^{1,2}_0(U)} :=\|\cdot \|_{H^{1,2}(U)}$.
We write $2_*:= \frac{2d}{d+2}$ if $d \geq 3$ and $2_* \in (1,2)$ is arbitrary but fixed if $d=2$. Let $2^*= \frac{2d}{d-2}$ if $d \geq 3$ and $2^* =\frac{2_*}{2_*-1}$ if $d=2$.
If $\mathcal{A}$ is a set of Lebesgue measurable functions, we define 
$\mathcal{A}_0:= \{ f  \in \mathcal{A}: \text{supp}(|f| dx) \text{ is compact in $U$}   \}$ and $\mathcal{A}_b := \mathcal{A} \cap L^{\infty}(U)$. 
\\
From now, we assume $d \geq 2$. 
Here we introduce basic knowledge of Dirichlet form theory.
Assume that $U$ is a bounded open subset of $\mathbb{R}^d$ and \eqref{assump} holds.
Consider a bilinear form $(\mathcal{E}^0, C_0^{\infty}(U))$ defined by
$\mathcal{E}^0(f,g) = \int_{U} \langle  A \nabla f, \nabla g \rangle dx$, $f,g \in C_0^{\infty}(U)$,
where $\nabla f = (\partial_1 f, \ldots, \partial_d f)$.
Then, as a consequence of the results in \cite[II, 2, b)]{MR}, $(\mathcal{E}^0, C_0^{\infty}(U))$ is closable on $L^2(U)$, i.e. for any sequence of functions $(u_n)_{n \geq 1}$ in $C_0^{\infty}(U)$ satisfying that $\lim_{n,m \rightarrow \infty}\mathcal{E}^0(u_n-u_m, u_n-u_m)=0$ and $\lim_{n \rightarrow \infty} u_n =0$ in $L^2(U)$, it holds $\lim_{n \rightarrow \infty} \mathcal{E}^0(u_n, u_n)=0$. Denote by $(\mathcal{E}^0, D(\mathcal{E}^0))$ the closure of $(\mathcal{E}^0, C_0^{\infty}(U))$ on $L^2(U)$. Let $\mathcal{E}^0_{\alpha}(u,v) = \mathcal{E}^0(u,v)+\alpha \int_{U} uv dx$, $u, v \in D(\mathcal{E}^0)$ for $\alpha>0$. Indeed, as a consequence of the fact that $A$ is uniformly elliptic and bounded, we obtain that $\mathcal{E}^0_1(\cdot, \cdot)^{1/2}$ is equivalent to $\| \cdot \|_{H^{1,2}_0(U)}$ and that $D(\mathcal{E}^0) = H^{1,2}_0(U)$. In particular, by using the Cauchy-Schwarz inequality, for any $f,g \in D(\mathcal{E}^0)$,
$\left| \mathcal{E}^0(f,g) \right| \leq  d \Lambda \lambda^{-1} \mathcal{E}^0(f,f)^{1/2}  \mathcal{E}^0(g,g)^{1/2}$, and hence
 $(\mathcal{E}^0, D(\mathcal{E}^0))$  is a coercive closed form (see \cite[I. Definition 2.4]{MR}). According to the contents in \cite[I, 1.2]{MR}, denote by $(G^0_{\alpha})_{\alpha>0}$, $(T^0_t)_{t>0}$ and $(L^0, D(L^0))$ the strongly continuous contraction resolvent on $L^2(U)$ (called {\it $C_0$-resolvent of contractions on $L^2(U)$}),  the strongly continuous contraction semigroup on $L^2(U)$ (called {\it $C_0$-semigroup of contractions on $L^2(U)$}) and the generator on $L^2(U)$ which are associated with $(\mathcal{E}^0, D(\mathcal{E}^0))$, respectively.
On the other hand, it follows from \cite[I. Proposition 4.7]{MR} and the explanation after \cite[II. (2.18)]{MR} that  $(\mathcal{E}^0, D(\mathcal{E}^0))$  is a Dirichlet form (see \cite[I. Definition 4.5]{MR}). Hence, by \cite[I. Proposition 4,3, Theorem 4.4]{MR} $(G^0_{\alpha})_{\alpha>0}$ and $(T^0_t)_{t>0}$ are a sub-Markovian $C_0$-resolvent and a sub-Markovian $C_0$-semigroup of contractions on $L^2(U)$, respectively.

\begin{proposition} \label{prop1.1}
Assume that $U$ is a bounded open subset of $\mathbb{R}^d$ and \eqref{assump} holds.  Then, there exists a sub-Markovian $C_0$-semigroup $(\overline{T}^0_t)_{t>0}$ (resp. sub-Markovian $C_0$-resolvent $(\overline{G}_{\alpha}^0)_{\alpha>0}$) of contractions on $L^1(U)$ such that for each $f \in L^{2}(U)$, $\overline{T}^0_t f = T^0_t f$ for all $t>0$  (resp. $\overline{G}^0_{\alpha} f = G^0_{\alpha}f$ for all $\alpha>0$). Let $(\overline{L}^0, D(\overline{L}^0))$ be the generator of $(\overline{T}^0_t)_{t>0}$ on $L^1(U)$. Then, $(L^0, D(L^0)) \subset (\overline{L}^0, D(\overline{L}^0))$, i.e. $D(L^0) \subset D(\overline{L}^0)$ and $\overline{L}^0 u = L^0 u$ for all $u \in D(L^0)$. 
\end{proposition}

\begin{proof}
As a consequence of the results in \cite[II, 2, b)]{MR}, $(\hat{\mathcal{E}}^0, C_0^{\infty}(U))$ defined by
$\hat{\mathcal{E}}^0(f,g) = \int_{U} \langle  A^T \nabla f, \nabla g \rangle dx$, $f,g \in C_0^{\infty}(U)$, is closable.  Moreover, 
by \cite[I. Proposition 4.7]{MR}, the explanation after \cite[II. (2.18)]{MR} and \cite[I. Proposition 4,3, Theorem 4.4]{MR}, the closure of $(\hat{\mathcal{E}}^0, C_0^{\infty}(U))$ denoted by  $(\hat{\mathcal{E}}^0, D(\hat{\mathcal{E}}^0))$ is a Dirichlet form which generates a sub-Markovian $C_0$-semigroup of contractions $(\hat{T}^0_t)_{t>0}$ on $L^2(U)$. Then, by \cite[I. Theorem 2.8]{MR}, $(T^0_t)_{t>0}$ is the adjoint of $(\hat{T}^0_t)_{t>0}$ on $L^2(U)$. Thus, for each $t>0$ and $f \in L^{2}(U)$
\begin{align} \label{l1contra}
\int_{U}|T^0_t f| dx \leq \lim_{n \rightarrow \infty} \int_{U} T^0_t |f| \cdot 1_{D_n} dx = \lim_{n \rightarrow \infty} \int_{U} |f| \cdot \hat{T}^0_t 1_{D_n} dx \leq \int_{U} |f| dx,
\end{align}
where $D_n=\{ x \in \mathbb{R}^n:  \|x\|<n  \}$.
Since $L^{2}(U)$ is dense in $L^1(U)$, \eqref{l1contra} implies that $(T^0_t)_{t>0}$ uniquely extends to a sub-Markovian semigroup of contractions on $L^1(U)$, say $(\overline{T}_t^0)_{t>0}$. Since $\|\overline{T}^0_t f -f\|_{L^1(U)} \leq |U|^{1/2} \|T^0_t f -f\|_{L^2(U)}$  for all $t>0$ and $f \in L^{2}(U)$,
the $L^1$-contraction property of $(\overline{T}^0_t)_{t>0}$ and the strong continuity of $(T^0_t)_{t>0}$ on $L^2(U)$ imply that
$(\overline{T}^0_t)_{t>0}$ is strongly continuous on $L^1(U)$. Let $u \in D(L^0)$. Then, $u, L^0 u \in L^2(U)$, so that
$\overline{T}^0_t u = T_t^0 u$ and $\overline{T}^0_t L^0 u = T_t^0 L^0 u$ for every $t>0$. Using the strong continuity of $(\overline{T}^0_t)_{t>0}$ on $L^1(U)$, 
$$
\frac{1}{t}(\overline{T}^0_t u - u)=\frac{1}{t}(T^0_t u - u) = \frac{1}{t} \int_0^t T_s^0 L^0 u\, ds = \frac{1}{t} \int_0^t \overline{T}_s^0 L^0 u\, ds \rightarrow L^0 u  \;\;\text{ in $L^1(U)$}
$$
as $t \rightarrow 0+$. Thus, $u \in D(\overline{L}^0)$ and $\overline{L}^0 u = L^0 u$. 
\end{proof}
Hereafter, as in Proposition \ref{prop1.1} $(\overline{G}^0_{\alpha})_{\alpha>0}$ and $(G^0_{\alpha})_{\alpha>0}$ denote the sub-Markovian $C_0$-resolvents of contractions associated with $(\overline{T}_t^0)_{t>0}$ and $(T_t^0)_{t>0}$ on $L^1(U)$ and $L^2(U)$, respectively. $(L^0, D(L^0))$ denotes the generator on $L^2(U)$ associated with $(G^0_{\alpha})_{\alpha>0}$, while $(\overline{L}^0, D(\overline{L}^0))$ denotes the generator on $L^1(U)$ associated with $(\overline{G}^0_{\alpha})_{\alpha>0}$.

\begin{lemma} \label{boleulem}
Assume that $U$ is a bounded open subset of $\mathbb{R}^d$ and \eqref{assump} holds.
Let $u \in H^{1,2}_0(U)_b$ and $w \in L^1(U)$. Assume that  
$$
\mathcal{E}^0(u, v) =-\int_{U} w v dx \;\;\text{ for all $v \in H_0^{1,2}(U)_b$.}
$$
Then, $u \in D(\overline{L}^0)_b$ and $\overline{L}^0u =w$.
\end{lemma}
\begin{proof}
Let $(\hat{\mathcal{E}}^0, D(\hat{\mathcal{E}}^0))$ be the Dirichlet form as in the proof of Proposition \ref{prop1.1} and denote its corresponding sub-Markovian $C_0$-resolvent of contractions on $L^2(U)$ by $(\hat{G}_{\alpha})_{\alpha>0}$. Let $\varphi \in C_0^{\infty}(U)$. Then, by the condition and \cite[I. Theorem 2.8]{MR},
$$
-\int_{U} \varphi \overline{G}^0_1 w dx = -\int_{U} w \hat{G}^0_1 \varphi dx = \mathcal{E}^0(u, \hat{G}^0_1 \varphi)  = \int_{U} u(\varphi  -  \hat{G}^0_1 \varphi) dx = \int_{U} (u-\overline{G}^0_1u) \varphi dx.
$$
Thus, $u =\overline{G}^0_1(u-w) \in D(\overline{L})_b$ and $(1-\overline{L}^0)u = u-w$, as desired.
\end{proof}

\begin{lemma} \label{lem1.1}
Assume that $U$ is a bounded open subset of $\mathbb{R}^d$ and \eqref{assump} holds. Then the following hold:
\begin{itemize}
\item[(i)]
$D(\overline{L}^0)_b \subset H^{1,2}_0(U)$ and $\lim_{\alpha \rightarrow \infty} \alpha G^0_{\alpha} u = u$ in $H^{1,2}_0(U)$ for any $u \in D(\overline{L}^0)_b$. Moreover, for all $u \in D(\overline{L}^0)_b$ and $v \in H^{1,2}_0(U)_b$, $\mathcal{E}^0(u,v) = -\int_{U}  \overline{L}^{0} u  \cdot v \,dx$.
\item[(ii)]
If $\psi \in C^2(\mathbb{R})$ with $\psi(0)=0$ and $u \in D(\overline{L}^0)_b$, then $\psi (u) \in D(\overline{L}^0)_b$ and
\begin{equation} \label{(1.1)}
\overline{L}^0  \psi(u) = \psi'(u) \,\overline{L}^0 u + \psi''(u) \langle A \nabla u, \nabla u \rangle. 
\end{equation}
\end{itemize}
\end{lemma}

\begin{proof}
(i) Let $u \in D(\overline{L}^0)_b$. Since $u \in L^{\infty}(U)$, it holds $\overline{G}^0_{\alpha} u=G^{0}_{\alpha} u  \in D(L^0)$ for any $\alpha>0$. And since $(L^0, D(L^0)) \subset (\overline{L}^0, D(\overline{L}^0))$ by Proposition \ref{prop1.1}, it follows that  $L^0 G_{\alpha}^0 u =\overline{L}^0 G_{\alpha}^0 u = \overline{L}^0 \overline{G}_{\alpha}^0 u = \overline{G}_{\alpha}^0 \overline{L}^0  u$ for any $\alpha>0$. Moreover, it follows from \cite[Corollary 2.10]{MR}, the sub-Markovian property of $(G^0_{\alpha})_{\alpha>0}$ and the $L^1(U)$-strong continuity of  $(\overline{G}^0_{\alpha})_{\alpha>0}$
 that
\begin{align*}
&\mathcal{E}^0 (\alpha G^0_{\alpha} u - \beta G^0_{\beta} u, \alpha G^0_{\alpha} u - \beta G^0_{\beta} u ) = - \int_{U} L^0 \big( \alpha G^0_{\alpha} u - \beta G^0_{\beta} u \big) \cdot (\alpha G^0_{\alpha} u - \beta G^0_{\beta} u) dx \\
& \qquad =  - \int_{U} \big( \alpha \overline{G}_{\alpha}^0 \overline{L}^0  u- \beta \overline{G}_{\alpha}^0 \overline{L}^0  u \big) \cdot (\alpha G^0_{\alpha} u - \beta G^0_{\beta} u) dx \\
&\qquad \leq  2 \|u\|_{L^{\infty}(U)} \| \alpha \overline{G}_{\alpha}^0 \overline{L}^0  u- \beta \overline{G}_{\alpha}^0 \overline{L}^0  u \|_{L^1(U)} \rightarrow 0  \text{ as }\alpha, \beta \rightarrow \infty.
\end{align*}
Thus, $(\alpha G^0_{\alpha} u)_{\alpha>0}$ is a Cauchy sequence in $H_0^{1,2}(U)$, and hence  by the $L^2(U)$-strong continuity of $(G^0_{\alpha})_{\alpha>0}$, we obtain $u \in H_0^{1,2}(U)$ and $\lim_{\alpha \rightarrow \infty}\alpha G^0_{\alpha} u = u$ in $H^{1,2}_0(U)$.  Moreover,  using \cite[I. Corollary 2.10]{MR} and the $L^1(U)$-strong continuity of $(\overline{G}^0_{\alpha})_{\alpha>0}$, for any $u \in D(\overline{L}^0)_b$ and $v \in H^{1,2}_0(U)_b$,
\begin{align*}
\mathcal{E}^0(u,v)  &=\lim_{\alpha \rightarrow \infty} \mathcal{E}^0 (\alpha G^0_{\alpha} u, v) = \lim_{\alpha \rightarrow \infty}  -\int_{U} L^0 \big( \alpha G^0_{\alpha} u  \big) \cdot v dx \\
&= \lim_{\alpha \rightarrow \infty} -\int_{U}  \big( \alpha \overline{G}^0_{\alpha} \overline{L}^0 u  \big) \cdot v dx  = -\int_{U} \overline{L}^0 u \cdot v dx. 
\end{align*}
(ii) Since $u \in H^{1,2}_0(U)$ by (i), there exists a sequence of functions $(\widetilde{u}_n)_{n \geq 1}$ in $C_0^{\infty}(U)$ such that $\lim_{n \rightarrow \infty}\widetilde{u}_n=u$ in $H^{1,2}_0(U)$ and $\lim_{n \rightarrow \infty} \widetilde{u}_n =u$ a.e. Let $\phi \in C^{\infty}(\mathbb{R})_b$ be such that $\phi(t)=t$ for all $t \in [-\|u\|_{L^{\infty}(U)}-1, \|u\|_{L^{\infty}(U)}+1]$. Let $M_1=\|\phi\|_{L^{\infty}(\mathbb{R})}$ and $u_n=\phi(\widetilde{u}_n)$, $n \geq 1$. Then, $u_n \in C_0^{\infty}(U)$, $\|u_n\|_{L^{\infty}(U)} \leq M_1$  for all $n  \geq 1$ and $\lim_{n \rightarrow \infty} u_n =u$ a.e. Using the chain rule and Lebesgue's theorem, $\lim_{n \rightarrow \infty}u_n=u$ in $H^{1,2}_0(U)$. By the chain rule, $\psi (u) \in H_0^{1,2}(U)_b$ and $\nabla \psi (u) = \psi'(u) \nabla u$. Let $v \in H^{1,2}_0(U)_b$. As above, there exist a sequence of functions $(v_n)_{n \geq 1}$ in $C_0^{\infty}(U)$ and a constant  $M_2>0$ such that
$\|v_n\|_{L^{\infty}(U)} \leq M_2$  for all $n \geq  1$
and that $\lim_{n \rightarrow \infty}v_n=v$ in $H^{1,2}_0(U)$ and $\lim_{n \rightarrow \infty} v_n =v$ a.e.
Thus, $v_n \psi'(u_n) \in C_0^1(U)$ for each $n \geq 1$, $v \psi'(u) \in H^{1,2}_0(U)$ and $\lim_{n \rightarrow \infty} v_n \psi'(u_n) = v \psi'(u)$ in $H^{1,2}_0(U)$. Indeed, $\nabla (v \psi'(u)) = \psi'(u)\nabla v + v  \psi''(u) \nabla u$.
Using (i), we obtain that for each $v \in H^{1,2}_0(U)_b$, 
$\mathcal{E}^0(\psi(u), v) 
= -\int_{U} w v dx$, where $w= \overline{L}^0 u \cdot \psi' (u) + \psi''(u) \langle A \nabla u, \nabla u \rangle \in L^1(U)$ (see \cite[Proof of Lemma 1.2(iv)]{S99}). Thus, by Lemma \ref{boleulem}, $\psi(u) \in D(\overline{L})_b$ and \eqref{(1.1)} holds.
\end{proof}

\section{Constructing a sub-Markovian $C_0$-resolvent of contractions}\label{sec3}

\begin{theorem} \label{theo1.2}
Assume {\bf (A)}. Let
\begin{equation} \label{givenop}
L u := L^0 u - \langle B, \nabla u \rangle -cu, \quad \;  u\in D(L^0)_b,
\end{equation}
where $(L^0, D(L^0))$ is the generator as in Section \ref{sec2}. Then, the following hold:
\begin{itemize}
\item[(i)]
The operator $(L, D(L^0)_b)$ is densely defined on $L^1(U)$ and dissipative, hence closable.

\item[(ii)]
	The closure $(\overline{L}, D(\overline{L}))$ of $(L, D(L^0)_b)$ on $L^1(U)$ generates a $C_0$-resolvent of contractions $(\overline{G}_{\alpha})_{\alpha>0}$
and a $C_0$-semigroup of contractions $(\overline{T}_t)_{t>0}$ on $L^1(U)$.

\item[(iii)]
$(\overline{G}_{\alpha})_{\alpha>0}$ and $(\overline{T}_t)_{t>0}$  are sub-Markovian, i.e. for each $\alpha>0$ and $t>0$
$$
0 \leq \alpha \overline{G}_{\alpha}f \leq 1, \quad 0 \leq \overline{T}_t f \leq 1, \quad \text{ for any $f \in L^1(U)$ with $0 \leq f \leq 1$}.
$$
\end{itemize}
\end{theorem}

\begin{proof}
(i) First note that \eqref{givenop} is well-defined due to the inclusion $D(L^0) \subset D(\overline{L}^0) \subset H^{1,2}_0(U)_b$ by Proposition \ref{prop1.1} and Lemma \ref{lem1.1}(i).
Let $f \in L^{\infty}(U)$. Then $n G^0_n f \in D(L^0)_b$  for each $n \in \mathbb{N}$ by the sub-Markovian property of $(G^0_{\alpha})_{\alpha>0}$. By the strong continuity of  $(G^0_{\alpha})_{\alpha>0}$ on $L^1(U)$, $\lim_{n \rightarrow \infty}n G^0_n f = f$ in $L^1(U)$. Since $L^{\infty}(U)$ is dense in $L^1(U)$, $D(L^0)_b$ is dense in $L^1(U)$. 
To show $(L, D(L^0)_b)$ is dissipative on $L^1(U)$, we will first show the following claim:
\begin{equation} \label{(1.2)}
\int_U L u \cdot 1_{(1, \infty)}(u) \,dx \leq 0   \;\; \text{ for any } u \in D(L^0)_b.
\end{equation}
To show the claim, let $u \in D(L^0)_b$. For each $\varepsilon \in (0,1)$ choose a function $\eta_{\varepsilon} \in C^{\infty}(\mathbb{R})$ satisfying that $\eta_{\varepsilon}(t) \in [0,1]$, $\eta_{\varepsilon}(t)=0$ for all $t \leq 1$ and $\eta_{\varepsilon}(t)=1$ for all $t  \geq 1+ \varepsilon$. 
For each $\varepsilon \in (0,1)$, let
$\psi_{\varepsilon}(t)=\int_{-\infty}^t \eta_{\varepsilon}(s) ds$, $t \in \mathbb{R}$.
Then, $\psi_{\varepsilon} \in C^{\infty}(\mathbb{R})$ with $\psi_{\varepsilon}(t) \geq 0$ for all $t \in \mathbb{R}$ and $\psi_{\varepsilon}(0)=0$ for all $t \in (-\infty, 0]$ and $\varepsilon \in (0,1)$. Hence by Lemma \ref{lem1.1}(ii), we get $\psi_{\varepsilon}(u) \in D(\overline{L}^0)$ and
\begin{align*}
\int_{U} L^0 u \cdot \psi_{\varepsilon}' (u) \, dx &\leq \int_U L^0 u \cdot \psi'_{\varepsilon}(u) dx +\int_{U} \psi_{\varepsilon}''(u) \langle A \nabla u, \nabla u \rangle dx  \nonumber \\
& = \int_{U} \overline{L}^0 \psi_{\varepsilon}(u) dx = \lim_{t \rightarrow 0+} \int_{U} \frac{\overline{T}^0_t \psi_{\varepsilon}(u) - \psi_{\varepsilon}(u)}{t} dx \leq 0 \label{(1.3)}
\end{align*}
for all $\varepsilon \in (0,1)$, where the last inequality follows from the $L^1$-contraction property of $(\overline{T}^0_t)_{t>0}$.
Note that $\lim_{\varepsilon \rightarrow 0+} \psi'_{\varepsilon}(t) = 1_{(1, \infty) }(t)$ for each $t>0$ and $\|\psi'_{\varepsilon}(u)\|_{L^{\infty}(U)} \leq 1$ for all $\varepsilon \in (0,1)$.
By \eqref{weakdivne},
$
-\int_{U} \langle B, \nabla u \rangle  \psi_{\varepsilon}'(u) \, dx =-\int_{U} \langle B, \nabla \psi_{\varepsilon}(u) \rangle dx \leq 0.
$ 
And it holds that $-\int_{U} c u \psi_{\varepsilon}'(u)  \,dx \leq 0$.
Therefore,
\begin{equation} \label{eq1.4}
\int_U L u \cdot \psi_{\varepsilon}'(u) \,dx = \int_{U} L^0 u \cdot \psi_{\varepsilon}'(u) - \langle B, \nabla u \rangle \psi_{\varepsilon}'(u) - c u \psi_{\varepsilon}'(u)\,dx  \leq 0.
\end{equation}
Letting $\varepsilon \rightarrow 0+$ in \eqref{eq1.4}, the claim follows from Lebesgue's theorem. \\
For each $n \in \mathbb{N}$, replacing $u \in D(L^0)_b$ by $n u \in D(L^0)_b$ in \eqref{(1.2)}, we have
$$
n \int_U L u \cdot 1_{(\frac{1}{n}, \infty)}(u) \,dx = \int_U L (nu) \cdot 1_{(1, \infty)}(nu) \,dx \leq 0 \quad \text{ for all $u \in D(L^0)_b$}.
$$
Dividing the above by $n$ and letting $n \rightarrow \infty$, we obtain from Lebesgue's theorem 
\begin{equation} \label{eq1.5}
\int_U L u \cdot 1_{(0, \infty)}(u) \,dx \leq 0 \quad \text{ for all $u \in D(L^0)_b$}.
\end{equation}
Replacing $u \in D(L^0)_b$ by $-u \in D(L^0)_b$ in \eqref{eq1.5}, it follows that
$$
-\int_U L u \cdot 1_{(-\infty, 0)}(u) \,dx = \int_U L (-u) \cdot 1_{(0, \infty)}(-u) \,dx \leq 0 \quad \text{ for all $u \in D(L^0)_b$},
$$
and hence $\int_U L u \cdot \big( 1_{(0, \infty)}(u) - 1_{(-\infty, 0)}(u) \big) \,dx \leq 0$ for all $u \in D(L^0)_b$. Since $L^1(U)'$ is identified as $L^{\infty}(U)$ and $ 1_{(0, \infty)}(u) - 1_{(-\infty, 0)}(u) \in L^{\infty}(U)$ satisfies  $\| 1_{(0, \infty)}(u) - 1_{(-\infty, 0)}(u) \|_{L^{\infty}(U)}=1$ and
$\int_{U}  u\big(  1_{(0, \infty)}(u) - 1_{(-\infty, 0)}(u) \big)\, dx  = \int_{U} u^+ + u^-\, dx =\|u\|_{L^1(U)}$, it follows that $(L, D(L^0)_b)$ is dissipative (see \cite[Definition 3.4.1]{AB11}). \\
(ii) Note that the closure $(\overline{L}, D(\overline{L}))$ of $(L, D(L^0)_b)$ on $L^1(U)$ is also dissipative by \cite[Lemma 3.4.4]{AB11}. To show assertion (ii), it is enough to show by the Lumer-Phillips Theorem (\cite[Theorem 3.4.5]{AB11}) that $(1-\overline{L})(D(\overline{L})) =L^1(U)$. Since $(1-\overline{L})(D(\overline{L}))$ is the closure of $(1-L)(D(L^0)_b)$ in $L^1(U)$, it is now enough to show the claim that $(1-L)(D(L^0)_b)$ is  dense in  $L^1(U)$. 
Let $h \in L^{\infty}(U)$ satisfy  that 
\begin{equation} \label{1.5eq}
\int_{U} (1-L) u \cdot h dx =0 \text{ for all $u \in D(L^0)_b$. }
\end{equation}
Then, to show the claim, it suffices to show that $h=0$ 
by a consequence of the Hahn-Banach theorem (see \cite[Proposition 1.9]{B11}). Note that the map, $\mathcal{T}: H^{1,2}_0(U) \rightarrow \mathbb{R}$ defined by $\mathcal{T}(u) := \int_{U} \langle -B, \nabla u\rangle h  dx - \int_{U} cuh dx$, $u \in H^{1,2}_0(U)$ is continuous with respect to the norm $\|\cdot \|_{H^{1,2}_0(U)}$. Thus, it follows from the Lax-Milgram theorem (\cite[Corollary 5.8]{B11}) that there exists $v \in H^{1,2}_0(U)$ such that
$\mathcal{E}_1^0(u,v)  = \mathcal{T}(u)$ for all $u \in H^{1,2}_0(U)$.
Since $D(L^0) \subset D(\mathcal{E}^0)=H^{1,2}_0(U)$  (see \cite[I. Corollary 2.10]{MR} and Section \ref{sec2}), we obtain that for each $u \in D(L^0)_b$
$$
\int_{U} (1-L^0) u \cdot v \, dx =\mathcal{E}_1^0(u,v)  = \mathcal{T}(u) = \int_{U} (1-L^0)u \cdot h dx.
$$
Thus, 
\begin{equation} \label{(1.5)}
\int_{U} (1-L^0) u \cdot (v-h) dx =0  \quad \text{ for all $u \in D(L^0)_b$}.
\end{equation}
By the sub-Markovian property of $(G^{0}_{\alpha})_{\alpha>0}$, it follows that $L^{\infty}(U) \subset (1-L^0) (D(L^0)_b)$, so that $v-h=0$ by \eqref{(1.5)}.
Thus, $h = v \in H^{1,2}_0(U)_b$. Since $\lim_{n \rightarrow \infty} n G^0_n h = h$ in $H^{1,2}_0(U)$ and in $L^{2^*}(U)$ by \cite[I. Theorem 2.13(ii)]{MR} and Sobolev's inequality, it holds from \eqref{1.5eq} that
\begin{align*}
0 &\leq \mathcal{E}^0_1(h,h) = \lim_{n \rightarrow \infty} \mathcal{E}^0_1(n G_n^0 h, h) = \lim_{n \rightarrow \infty} \int_{U} (1-L^0) nG^0_n h  \cdot h dx  \\
&= \lim_{n \rightarrow \infty} - \int_{U} \langle B, \nabla nG^0_n h \rangle h dx  - \int_{U} c n G^0_n h \cdot h dx = \int_{U} -\frac12 \langle B, \nabla h^2 \rangle  - c h^2 dx \leq 0,
\end{align*}
so that $h=0$ as desired. \\
(iii) 
The proof is the same as \cite[Step 3, proof of Proposition 1.1(i)]{S99}.
\end{proof}

\begin{remark} \label{dirichdiffer}
In Step 2 of the proof of \cite[Proposition 1.1(i)]{S99}, the fact that for each $h \in D(\mathcal{E}^{0})$, 
$
\lim_{t \rightarrow 0+} \mathcal{E}^{0} (T^{0,V}_t h -h, T^{0,V}_t h -h  )=0
$ 
(cf. \cite[Lemma 1.3.3]{FOT}) is used, where $(\mathcal{E}^0, D(\mathcal{E}^0))$ in \cite{S99} is a symmetric Dirichlet form. 
But in our case, $(\mathcal{E}^0, D(\mathcal{E}^0))$ is possibly a non-symmetric Dirichlet form, and hence  in the last part of the proof of Theorem \ref{theo1.2}(ii), we use the fact from \cite[I. Theorem 2.13(ii)]{MR} that $\lim_{n \rightarrow \infty} \mathcal{E}^0 ( nG^0_n h-h, nG^0_n h-h ) = 0$ for each $h \in D(\mathcal{E}^0)$.
\end{remark}

\begin{lemma} \label{lem1.3}
Assume {\bf (A)}. Let $(\overline{L}, D(\overline{L}))$ be as in Theorem \ref{theo1.2}. Let $u \in D(\overline{L})_b$ and $(u_n)_{n \geq 1} \subset D(L^0)_b$ be a sequence of functions such that $u_n \rightarrow u$ in $D(\overline{L})$ and $u_n \rightarrow u$ a.e. 
Let $M_1, M_2>0$ be constants satisfying that 
$\|u\|_{L^{\infty}(U)}<M_1<M_2$.
Then,
\begin{equation} \label{(1.10)}
\lim_{n \rightarrow \infty} \int_{ \{M_1 \leq |u_n| \leq M_2\}} \langle A \nabla u_n, \nabla u_n \rangle dx = 0.
\end{equation}
\end{lemma}

\begin{proof}
Let  $M_3:=M_2-M_1$ and $\eta(t) := (t-M_1)^+  \wedge M_3$, $t \in \mathbb{R}$. Then, $\eta(t)=0$ for all $t \in (-\infty, M_1]$
 and $\eta(t) \geq 0$ for all $t \in \mathbb{R}$. Let $\phi(t):= \int_{-\infty}^{t}  \eta(s) ds$, $t \in \mathbb{R}$. Then, $\phi \in C^1(\mathbb{R})$ with $\phi(t) \geq 0$ for all $t \in \mathbb{R}$. Let $v \in H^{1,2}_0(U)$. Then, $\phi(v) \in H^{1,2}_0(U)$ and $\nabla \phi(v) = \eta(v) \nabla v$ by the chain rule. Moreover, by \cite[Theorem 4.4(iii)]{EG15} 
$$
\eta(v)  =(v-M_1)^+ \wedge M_3 = M_3-\left((v-M_1)^+ -M_3 \right)^- \in H^{1,2}_0(U),
$$ 
and that $\nabla \eta(v) = 1_{ \{ M_1 < v <M_2\} }\nabla v$. Indeed, since $\nabla v =0$ on $\{v= M\}$ for each $M \in \mathbb{R}$ (\cite[Theorem 4.4(iv)]{EG15}), it follows that
$\nabla \eta(v) = 1_{ \{ M_1 \leq  v  \leq M_2\} }\nabla v$. Thus, 
\begin{align*}
&0 \leq \int_{\{ M_1 \leq u_n \leq M_2 \}} \langle A \nabla u_n, \nabla u_n \rangle dx = \int_{U} \langle A \nabla u_n, 1_{\{ M_1 \leq u_n \leq M_2 \}}\nabla u_n \rangle dx \\
& =  \int_{U} \langle A \nabla u_n, \nabla \eta(u_n) \rangle dx  = -\int_{U} L^0 u_n \cdot \eta(u_n) dx  \\
&\leq  -\int_{U} L^0 u_n \cdot \eta(u_n) dx + \int_{U} \langle B, \nabla \phi(u_n) \rangle dx  + \int_{U} cu_n \eta(u_n) dx   \\\
& = - \int_{U} \overline{L} u_n \cdot \eta(u_n) dx \rightarrow   - \int_{U} \overline{L} u \cdot \eta(u)  dx =0 \;\; \text{ as $n \rightarrow \infty$},
\end{align*}
where the last line follows from  Lebesgue's theorem and the fact that $\eta(u)=0$. Replacing $u_n$ by $-u_n$ in the above,
\eqref{(1.10)} follows.
\end{proof}

\begin{theorem} \label{theo1.4}
Assume {\bf (A)}. Let $(\overline{L}, D(\overline{L}))$ be as in Theorem \ref{theo1.2}. Then, the following hold:
\begin{itemize}
\item[(i)]
$D(\overline{L}^0)_b \subset D(\overline{L})$ and $\overline{L} u  = \overline{L}^0 u -\langle B, \nabla u \rangle -c u$\, for all $u \in D(\overline{L}^0)_b$.

\item[(ii)]
$D(\overline{L})_b \subset H^{1,2}_0(U)$ and for all $u \in D(\overline{L})_b$ and $v \in H^{1,2}_0(U)_b$,
\begin{equation} \label{eq.1.11}
\mathcal{E}^0(u, v) + \int_{U} \langle B, \nabla u \rangle v + cu v \,dx = - \int_{U} \overline{L} u \cdot v \, dx.
\end{equation}

\item[(iii)]
For any $f \in L^{\infty}(U)$, $v \in H^{1,2}_0(U)_b$ and $\alpha \in (0, \infty)$, 
\begin{equation} \label{resoliden} 
\int_{U} \langle A \nabla \overline{G}_{\alpha}f, \nabla v \rangle+  \langle B, \nabla  \overline{G}_{\alpha} f \rangle v + (c+\alpha)\overline{G}_{\alpha}f  \cdot v \,dx =  \int_{U}  f v \, dx.
\end{equation}
\end{itemize}
\end{theorem}

\begin{proof}
(i)
Let $u \in D(\overline{L}^0)_b$. Then $u \in H^{1,2}_0(U)$ by Lemma \ref{lem1.1}(i) and
$n G^0_{n} u \in D(L^0)_b \subset D(\overline{L}) \cap H^{1,2}_0(U)$. Since $(\overline{G}_{\alpha})_{\alpha>0}$ is
strongly continuous on $L^1(U)$ and $nG^0_{n} u \rightarrow u$ in $H^{1,2}_0(U)$ and in $L^{2^*}(U)$ as $n \rightarrow \infty$ by Lemma \ref{lem1.1}(i) and Sobolev's inequality, we have
\begin{align*}
L (nG^0_{n} u)  &= L^0 (nG^0_{n} u)  - \langle B, \nabla nG_n^0 u \rangle -c nG^0_n u, \\
&= n \overline{G}^0_n\, \overline{L}^0 u - \langle B, \nabla nG_n^0 u \rangle -c nG^0_n u \rightarrow  \overline{L}^0 u -\langle B, \nabla u \rangle -c u \;\text{ in $L^1(U)$}.
\end{align*}
By the closedness of $(\overline{L}, D(\overline{L}))$ on $L^1(U)$, $u \in D(\overline{L})$ and the assertion follows. \\
(ii) Let $M_1 =\|u\|_{L^{\infty}(U)}+1$, $M_2=\|u\|_{L^{\infty}(U)}+2$ and $\psi \in C_0^{2}(\mathbb{R})$ be such that $\psi(t)=t$ for all $|t| \leq M_1$ and $\psi(t)=0$ for all $|t| \geq M_2$.  Let $(u_n)_{n \geq 1} \subset D(L^0)_b$ be such that $u_n \rightarrow u$ in $D(\overline{L})$ and $u_n \rightarrow u$ a.e. Then by Lemma \ref{lem1.1}(ii) and Theorem \ref{theo1.4}(i), $\psi(u_n) \in D(\overline{L}^0)_b \cap H^{1,2}_0(U) \subset D(\overline{L})$ and
\begin{align*}
\overline{L} \psi(u_n)&=\overline{L}^0 \psi(u_n) -\langle B, \nabla \psi(u_n) \rangle -c \psi(u_n) \\
&= \psi' (u_n) \left(L u_n +c \psi(u_n) \right)- c \psi(u_n)  + \psi''(u_n) \langle A \nabla u_n, \nabla u_n \rangle.
\end{align*}
By Lebesgue's theorem, $\lim_{n \rightarrow \infty}\psi' (u_n) \big(L u_n +c \psi(u_n) \big)- c \psi(u_n) =\overline{L}u$ in $L^1(U)$.
By Lemma \ref{lem1.3},
\begin{align*}
&\int_{U} |\psi''(u_n) \langle A \nabla u_n, \nabla u_n \rangle| dx \leq \max_{[M_1, M_2]}|\psi''| \int_{ \{M_1 \leq |u_n| \leq M_2\}}  \langle A \nabla u_n, \nabla u_n \rangle dx \rightarrow 0.  
\end{align*} 
 as $n \rightarrow \infty$, so that $\lim_{n \rightarrow \infty} L \psi(u_n)  = \overline{L}u$ in $L^1(U)$. Meanwhile, by Lemma \ref{lem1.1}(i) and Theorem \ref{theo1.4}(i) for any $v \in D(\overline{L}^0)_b$,
\begin{align*}
\mathcal{E}^0(v,v)=-\int_{U}  v\overline{L}^{0} v \,dx = - \int_{U} v \overline{L} v+ \frac12 \langle B, \nabla v^2 \rangle +cv^2 \, dx \leq \|v\|_{L^{\infty}(U)} \| \overline{L}v \|_{L^1(U)}.
\end{align*}
Hence, $\mathcal{E}^0\big( \psi(u_n) - \psi(u_m), \psi(u_n)- \psi(u_m)    \big)  \leq 2 \| \psi \|_{L^{\infty}(\mathbb{R})} \left\|   \overline{L} \psi(u_n)  - \overline{L} \psi(u_m) \right\|_{L^1(U)} \rightarrow 0$ as $n, m \rightarrow \infty$. Thus, by the completeness of $H^{1,2}_0(U)$ and Sobolev's inequality, $u \in H^{1,2}_0(U) \cap L^{2^*}(U)$ and $\lim_{n \rightarrow \infty} \psi(u_n) = \psi(u)$ in $H^{1,2}_0(U)$ and in $L^{2^*}(U)$. Therefore, by Lemma \ref{lem1.1}(i) and Theorem \ref{theo1.4}(i) for any $v \in H^{1,2}_0(U)_b$ 
\begin{align*}
&\mathcal{E}^0(u,v)  + \int_{U} \langle B, \nabla u \rangle v + cuv dx \\
&=\lim_{n \rightarrow \infty} \mathcal{E}^0(\psi(u_n),v)  + \int_{U} \langle B, \nabla \psi(u_n) \rangle v dx  +\int_{U} c\psi(u_n)v dx \\
&= \lim_{n \rightarrow \infty} -\int_{U} \overline{L} \psi(u_n) \cdot v dx =  -\int_{U} \overline{L} u \cdot v dx.
\end{align*} 
(iii) By substituting $u$ for $G_{\alpha} f$, $f \in L^{\infty}(U)$ and $\alpha>0$ in \eqref{eq.1.11}, the assertion follows.
\end{proof}

\begin{definition}
Let $(\overline{G}_{\alpha})_{\alpha>0}$ be the sub-Markovian $C_0$-resolvent of contractions on $L^1(U)$ as in Theorem \ref{theo1.2}(ii). Then, by a consequence of the Riesz-Thorin interpolation (\cite[Chapter 2, Theorem 2.1]{S11}), $(\overline{G}_{\alpha})_{\alpha>0}$ restricted on $L^{\infty}(U)$ extends to a sub-Markovian $C_0$-resolvent of contractions on $L^r(U)$ for each $r \in [1, \infty)$ and  a sub-Markovian resolvent on $L^{\infty}(U)$. Denote these by $(G_{\alpha})_{\alpha>0}$ independently of the $L^r(U)$-space, $r \in [1, \infty]$ on which they are acting. Indeed, $G_{\alpha}f = \overline{G}_{\alpha}f$  and by the $L^r$-contractions properties, 
\begin{equation} \label{contraprop}
\|G_{\alpha}f\|_{L^r(U)} \leq \alpha^{-1}\|f\|_{L^r(U)}
\end{equation}
for any $\alpha>0$ and $f \in L^r(U)$ with $r \in [1, \infty]$.

\end{definition}

\section{Existence of bounded weak solutions}\label{sec4}

\begin{theorem} \label{theo2.6}
Assume {\bf (A)}.  Let $f \in L^{2_{*}}(U)$ and $F \in L^2(U, \mathbb{R}^d)$.
Then, the following hold:

\begin{itemize}
\item[(i)]
Let $u \in H^{1,2}_0(U)_b$ be given. Assume that  $u$ is a weak solution to \eqref{undeq}.
Then, \eqref{enerestim} holds.

\item[(ii)]
Let $w \in H^{1,2}_0(U)_b$ be given. Assume that $w$ is a weak solution to \eqref{undeqdual}.
Then, \eqref{enerestim} is satisfied where $u$ is replaced by $w$.

\item[(iii)]
There exists a weak solution $u$  to \eqref{undeq} such that \eqref{enerestim} holds.

\item[(iv)]
There exists a weak solution $w$ to \eqref{undeqdual} such that \eqref{enerestim} is satisfied where $u$ is replaced by $w$.

\item[(v)]
Let $\alpha \in (0, \infty)$. Then, $G_{\alpha} f$ is a weak solution to \eqref{undeq} where $F$ is replaced by $0$. Moreover, \eqref{enerestim} holds where $u$ and $F$ are replaced by $G_{\alpha}f$ and $0$, respectively. If $r \in [1, \infty]$ and $f \in L^{2_*}(U) \cap L^r(U)$, then $G_{\alpha}f \in L^r(U)$ and \eqref{contraprop} holds.

\end{itemize}
\end{theorem}

\begin{proof}
(i) By an approximation, \eqref{(maineq)} holds for any $\varphi \in H_0^{1,2}(U)_b$. Substituting $u$ for $\varphi $, it follows from Young's inequality and Sobolev's inequality that
\begin{align*}
\lambda \| \nabla u \|^2_{L^2(U)} &\leq \int_{U} \langle A \nabla u, \nabla u \rangle dx \leq \int_{U} f u + \langle F, \nabla u \rangle dx \\
& \leq N^2 \varepsilon \|\nabla u\|^2_{L^2(U)} + (4\varepsilon)^{-1} \|f\|^2_{L^{2_*}(U)}  + \varepsilon \|\nabla u \|^2_{L^2(U)} + (4\varepsilon)^{-1}\|F\|_{L^2(U)}^2,
\end{align*}
where $N>0$ is the constant as in Corollary \ref{stabil}. Choosing $\varepsilon=\frac{\lambda}{2(N^2+1)}$ and using Sobolev's inequality, \eqref{enerestim} follows. \\
(ii) Choosing $\varphi=w$, the proof is analogous to the one of (i). \\
(iii) 
Let $\eta$ be a standard mollifier on $\mathbb{R}^d$ and define $\eta_{\varepsilon}(x):= \varepsilon^{-d}\eta(x/\varepsilon)$, $x \in \mathbb{R}^d$ and $\varepsilon>0$.
For $\varepsilon \in (0,1)$, let $U_{\varepsilon}:= \{ x \in U : \|x-y\|>\varepsilon \text{ for all $y \in \partial U$}\} $. Choose $\delta \in (0,1)$ so that $U_{\delta} \neq \emptyset$. Let $\hat{B}$ be a zero extension of $B$ on $\mathbb{R}^d$ and set $B_n = \hat{B} * \eta_{\delta/2n}$ for each $n \geq 1$. Then, $\text{div} B_n \leq 0$ weakly in $U_{\delta/2n}$ for each $n \geq 1$. Let $(c_n)_{n \geq 1}$ and  $(f_n)_{n \geq 1}$ be sequences of functions in $C_0^{\infty}(U)$ such that 
$\lim_{n \rightarrow \infty} c_n =c$ in $L^{2_{*}}(U)$ with $c_n \geq 0$ for all $n \geq 1$ and $\lim_{n \rightarrow \infty} f_n = f$ in $L^{2_*}(U)$. Let $(F_n)_{n \geq1}$ be a sequence of vector field in $C_0^{\infty}(U, \mathbb{R}^d)$ such that
$\lim_{n \rightarrow \infty} F_n = F$ in $L^2(U, \mathbb{R}^d)$. Let $n \geq 1$. Then, by \cite[Theorems 3.2, 4.1]{T73}, there exists $u_n \in H^{1,2}_0(U_{\delta/2n})_b$ such that \eqref{(maineq)} and \eqref{enerestim} hold where
$u$, $B$, $c$, $f$, $F$ and $U$ are replaced by $u_n$, $B_n$, $c_n$, $f_n$, $F_n$ and $U_{\delta/2n}$, respectively. Extend $u_n\in H_0^{1,2}(U_{\delta/2n})$ to $\widetilde{u}_n \in H^{1,2}_0(U)$ by the zero extension. Using the weak compactness of $H^{1,2}_0(U)$, there exist 
$u \in H^{1,2}_0(U)$  and a subsequence of $(\widetilde{u}_n)_{n \geq 1} \subset H^{1,2}_0(U)_b$, say again $(\widetilde{u}_n)_{n \geq 1}$ such that $\lim_{n \rightarrow \infty} \widetilde{u}_n = \widetilde{u}$ weakly in $H^{1,2}_0(U)$, and hence the assertion follows. \\
(iv) The proof is analogous to the one of (iii).\\
(v) Let $\alpha \in (0, \infty)$ and $(f_n)_{n \geq 1} \subset L^{\infty}(U)$ be such that $\lim_{n \rightarrow \infty} f_n =f$ in $L^{2_*}(U)$. By Theorem \ref{theo1.4}(iii), for each $n, m \in \mathbb{N}$ \eqref{(maineq)} holds where $\overline{G}_{\alpha}f$ and $f$ are replaced by $G_{\alpha}(f_n-f_m)$ and $f_n-f_m$, respectively. Hence for each $n, m \in \mathbb{N}$ \eqref{enerestim} holds where $\overline{G}_{\alpha}f$ and $f$ are replaced by $G_{\alpha}(f_n-f_m)$ and $f_n-f_m$, respectively. Since $H^{1,2}_0(U)$ is complete, $G_{\alpha} f \in H^{1,2}_0(U)$ and $\lim_{n \rightarrow \infty} G_{\alpha} f_n  = G_{\alpha} f$ in $H^{1,2}_0(U)$, and hence for any $\varphi \in C_0^{\infty}(U)$ \eqref{resoliden} holds where $\overline{G}_{\alpha} f$ is replaced by $G_{\alpha} f$. Moreover, \eqref{enerestim} is satisfied where $u$ and $F$ are replaced by $G_{\alpha}f$ and $0$, respectively. 
The rest follows from the $L^r$-contraction properties of $(G_{\alpha})_{\alpha>0}$. 
\end{proof}

\noindent The proof of the following is based on the method of Moser's iteration.
\begin{theorem}  \label{boundthm}
Assume {\bf (A)}. Let $f \in L^q(U)$ and $F \in L^{2q}(U, \mathbb{R}^d)$ with $q>\frac{d}{2}$ and $q \geq 2_*$. Then the following hold:
\begin{itemize}
\item[(i)]
Let $u \in H^{1,2}_0(U)_b$ be given. Assume that $u$ is a weak solution to \eqref{undeq}.
Then, \eqref{linfinestim} holds.

\item[(ii)]
Let $w \in H^{1,2}_0(U)_b$ be given. Assume that $w$ is a weak solution to \eqref{undeqdual}. Then, \eqref{linfinestim} holds where $u$ is replaced by $w$.
\end{itemize}
\end{theorem}

\begin{proof}
(i) First assume that $\|f\|_{L^q(U)}+ \|F\|_{L^{2q}(U)}=0$. Then substituting $u$ for $\varphi$, it follows that  $u=0$. Now assume that $\|f\|_{L^q(U)}+ \|F\|_{L^{2q}(U)}>0$. Let $k=\|f\|_{L^q(U)} + \|F\|_{L^{2q}(U)}$ and $\bar{t}=t^++k$, $t \in \mathbb{R}$.
Then, $\bar{u}=u^+ +k \in H_0^{1,2}(U)_b$ with $\nabla \bar{u}  = 1_{\{u>0\}} \,\nabla u$, a.e.  Let $\beta \geq 0$ and choose the test function $\varphi = \bar{u}^{\beta+1}-k^{\beta+1}$.
Then, by the chain rule, $\varphi \in H^{1,2}_0(U)_b$ and $\nabla \varphi = (\beta+1) \bar{u}^{\beta} \nabla \bar{u}$.
Note that
\begin{align}
\hspace{-0.2em} \int_{U } \langle A \nabla u, \nabla \varphi \rangle dx  &= (\beta+1)\int_{U} \bar{u}^{\beta} \langle A \nabla u, \nabla \bar{u} \rangle dx   = (\beta+1)\int_{U} \bar{u}^{\beta} \langle A \nabla u, \nabla {u} \rangle 1_{\{u>0\}} dx \nonumber \\
&= (\beta+1)\int_{U} \bar{u}^{\beta} \langle A \nabla \bar{u}, \nabla \bar{u} \rangle dx    \geq (\beta+1) \lambda \int_U \bar{u}^{\beta} \| \nabla \bar{u} \|^2 dx. \label{ellip}
\end{align}
Let $H(t)=\frac{1}{\beta+2} \left( \bar{t}^{\, \beta+2} - (\beta+2) \bar{t} k^{\beta+1} +(\beta+1) k^{\beta+2} \right)$, $t \in \mathbb{R}$. Then, $H \in C^1(\mathbb{R})$ and $H'(t)=\bar{t}^{\beta+1}-k^{\beta+1}$, $t \in \mathbb{R}$.
Since $H'(t) \geq 0$ for all $t \in \mathbb{R}$ and $H(t)=0$ for all $t \leq 0$, we obtain that $H(t) \geq 0$ for all $t \in \mathbb{R}$. Thus, 
$H(u) \in H_0^{1,2}(U)_b$ with $H(u) \geq 0$ in $U$ and that
$\nabla H(u)=(\bar{u}^{\beta+1}-k^{\beta+1})   \nabla u  = \varphi \nabla u$.
Thus, $\int_{U} \langle B, \nabla u \rangle \varphi dx = \int_{U} \langle B, \nabla H(u) \rangle  dx \geq 0$.
Also, $\int_{U} (c+\alpha) u \varphi dx \geq 0$. 
Now let  $v=\bar{u}^{\frac{\beta+2}{2}}$. Then, $v-k^{\frac{\beta+2}{2}} \in H^{1,2}_0(U)$ with $v-k^{\frac{\beta+2}{2}} \geq 0$ 
and $\nabla v = \frac{\beta+2}{2} \bar{u}^{\frac{\beta}{2}}  \nabla \bar{u}$ by the chain rule. Thus, it follows from \eqref{ellip} that
$\int_{U } \langle A \nabla u, \nabla \varphi \rangle dx  \geq (\beta+1) \lambda \left(2(\beta+1)^{-1} \right)^2 \int_U \|\nabla v \|^2 dx  
\geq 2 \lambda (\beta+2)^{-1}  \int_{U} \| \nabla v \|^2 dx. 
$
By the H\"{o}lder inequality, 
\begin{align*}
\int_{U} f \varphi dx &\leq \int_{U} |f| |\bar{u}|^{\beta+1} dx     \leq   \int_{U}  \left( k^{-1}|f| \right)  \bar{u}^{\beta+2} dx= \int_{U} ( k^{-1} |f|) v^2 dx  \nonumber \\
& \leq \| k^{-1} |f|  \|_{L^q(U)} \|v \|^2_{L^{\frac{2q}{q-1}}(U)} \leq \|v \|^2_{L^{\frac{2q}{q-1}}(U)}. 
\end{align*}
Likewise, by Young's and the H\"{o}lder inequalities,
\begin{align*}
\int_{U} \langle F, \nabla \varphi \rangle dx  &= \int_{U} (\beta+1) \overline{u}^{\beta} \langle F, \nabla \bar{u} \rangle  dx \leq \int_{U}  2k^{-1}\| F\| \cdot |v|  \cdot \|\nabla v\| \,dx \\
&\leq \lambda(\beta+2)^{-1} \int_{U} \| \nabla v \|^2 dx + \lambda^{-1}(\beta+2) \int_{U} k^{-2} \|F\|^2 |v|^2 dx \\
&\leq  \lambda(\beta+2)^{-1}\| \nabla v\|^2_{L^2(U)}+\lambda^{-1}(\beta+2) \| v\|^2_{L^{\frac{2q}{q-1}}(U)}.
\end{align*}
Thus, $\int_{U } \langle A \nabla u, \nabla \varphi \rangle dx \leq \int_{U} f \varphi + \langle F, \nabla \varphi \rangle dx$ implies that
\begin{equation} \label{equbd}
\lambda(\beta+2)^{-1}  \| \nabla v \|^2_{L^2(U)} \leq \big(1+\lambda^{-1}(\beta+2) \big)\|v\|^2_{L^{\frac{2q}{q-1}}(U)}. 
\end{equation}
Now let $d_0 =d$ if $d \geq 3$ and $d_0 =1+q \in (2, 2q)$ if $d=2$. Let $s=\frac{2 d_0}{d_0-2}$.
Using interpolation, triangle and Sobolev's inequalities, for any $\varepsilon>0$ 
\begin{align}
\|v\|_{L^{\frac{2q}{q-1}}(U)} & \leq \varepsilon \|v\|_{L^{s}(U)} +K_1 \varepsilon^{- \frac{d_0}{2q-d_0}} \|v\|_{L^2(U)} \nonumber \\
& \leq \varepsilon \|v-k^{\frac{\beta+2}{2}}\|_{L^{s}(U)} +\varepsilon k^{\frac{\beta+2}{2}}|U|^{1/s} +K_1 \varepsilon^{- \frac{d_0}{2q-d_0}} \|v\|_{L^2(U)}  \nonumber \\
& \leq N_1 \varepsilon  \| \nabla v \|_{L^2(U)}  +(\varepsilon |U|^{\frac{1}{s}-\frac12} +  K_1 \varepsilon^{- \frac{d_0}{2q-d_0}}) \|v\|_{L^2(U)}, \label{eq3}
\end{align}
where $K_1 = \frac{2q-d_0}{2q}\left(  \frac{2q}{d_0} \right)^{-\frac{d_0}{2q-d_0}}$ and
$N_1 =\frac{2(d-1)}{d-2}$ if $d \geq 3$ and $N_1=\frac{1}{2}s |U|^{\frac{1}{s}}$ if $d=2$. By choosing $\varepsilon =\frac{\lambda}{2N_1\sqrt{\lambda+1}} (\beta+2)^{-1}$ in \eqref{eq3} so that $(1+\lambda^{-1}(\beta+2)) \cdot 2N_1^2  \varepsilon^2 \leq \frac{1}{2} \lambda (\beta+2)^{-1} $, we obtain from \eqref{equbd} and \eqref{eq3} that
$\|\nabla v \|^2_{L^2(U)} \leq K_2(\beta+2)^{\theta} \| v\|^2_{L^2(U)}$,
where $\theta >1$ is a constant which only depends on $d$ and $q$ and $K_2>0$ is a constant which only depends on $d$, $q$, $\lambda$ and $|U|$. Thus, using Sobolev's inequality,
$$
\| v-k^{\frac{\beta+2}{2}} \|_{L^{2^*}(U)} \leq N \|\nabla (v-k^{\frac{\beta+2}{2}}) \|_{L^2(U)} \leq N K_2 ^{1/2}(\beta+2)^{\theta/2} \| v\|_{L^2(U)},
$$
where $N$ is the constant as in Corollary \ref{stabil}. Thus, by the triangle inequality
\begin{align*}
\|v\|_{L^{2^*}(U)} &\leq k^{\frac{\beta+2}{2}}|U|^{1/2^*}+ N K_2 ^{1/2}(\beta+2)^{\theta/2} \| v\|_{L^2(U)} \leq K_3  (\beta+2)^{\theta/2}\|v\|_{L^2(U)},
\end{align*}
where $K_3:=(|U|^{\frac{1}{2^*}-\frac12}+NK_2^{1/2})$. Thus,
\begin{equation} \label{iterat}
\|v^{\sigma}\|^{2/\sigma}_{L^2(U)}=\|v\|^2_{L^{2^*}(U)} \leq K^2_3 (\beta+2)^{\theta} \|v\|^2_{L^2(U)},
\end{equation}
where $\sigma=\frac{d}{d-2}$ if $d \geq 3$ and $\sigma=2^*/2$ if $d=2$. Now write $\gamma=\frac{\beta+2}{2} \geq 1$. Then, $v=\bar{u}^\gamma$ and \eqref{iterat} is rewritten as
\begin{equation} \label{aronestim}
\|\bar{u}^{\gamma \sigma} \|_{L^2(U)}^{2/ \sigma} \leq K_3^2 (2 \gamma)^{\theta} \|\bar{u}^{\gamma}\|^2_{L^2(U)}.
\end{equation}
Now for $m=0,1,2. \ldots, $ define  $\psi_m:=\|\bar{u}^{\sigma^m} \|_{L^2(U)}^{2/ \sigma^m}$.
Then, \eqref{aronestim} with $\gamma=\sigma^m$ implies that
$\psi_{m+1} \leq \Big( K_3^2 (2 \sigma^m)^{\theta}  \Big)^{1/\sigma^m} \psi_m =K_4^{m/\sigma^m} \psi_m$,
where $K_4 = (K_3^2 2^{\theta}+1) \sigma^{\theta}$. Thus,
$\psi_{m+1} \leq K_5 \psi_0$, where $K_5=K_4^{\sum_{j=0}^\infty j/ \sigma^j}$.
Therefore, using Theorem \ref{theo2.6}(i)
\begin{align} \label{final}
\| u^+ \|_{L^{\infty}(U)}&\leq \limsup_{m \rightarrow \infty} \psi_{m+1}^{1/2} \leq K_5^{1/2} \| \bar{u}\|_{L^2(U)}  \leq K_5^{1/2} (\|u\|_{L^2(U)}+k|U|^{1/2}),  \nonumber \\
&\leq K_5^{1/2} ( \|u\|_{H_0^{1,2}(U)}+k|U|^{1/2}) \leq K_6(\|f\|_{L^q(U)} + \|F\|_{L^{2q}(U)}),
\end{align}
where $K_6=C_1K_5^{1/2}(|U|^{\frac{1}{2_*}-\frac{1}{q}}+|U|^{\frac{1}{2}-\frac{1}{2q}} )+K_5^{1/2}|U|^{1/2}$ and $C_1>0$ is the constant as in Theorem \ref{theo2.6}(i). Replacing $u$ by $-u$ in \eqref{final}, \eqref{linfinestim} follows. \\
(ii) Choose $\varphi = \bar{w}^{\beta+1}-k^{\beta+1}$. Thus, it suffices to show that 
\begin{equation} \label{posidiv}
\int_{U} \langle B, w \nabla (\bar{w}^{\beta+1}-k^{\beta+1}) \rangle dx =(\beta+1) \int_{U} \langle B, (\bar{w}^{\beta+1}-k\bar{w}^\beta) \nabla \bar{w} \rangle dx \geq 0,
\end{equation}
where $\bar{w} = w^{+}+k$. Let $S(w):= \frac{1}{\beta+2} \bar{w}^{\beta+2} - \frac{1}{\beta+1}k\bar{w}^{\beta+1}+\frac{k^{\beta+2}}{(\beta+2)(\beta+1)} $. Then, similarly to the proof of (i), it holds that $S(w) \in H^{1,2}_0(U)_b$ with $S(w) \geq 0$ and that $\nabla S(w) = (\bar{w}^{\beta+1}-k\bar{w}^\beta) \nabla \bar{w}$. Thus, \eqref{posidiv} is shown. The rest is identical to the proof of (i).
\end{proof}

\begin{theorem} \label{theo3.2}
Assume {\bf (A)} and let $q>\frac{d}{2}$ with $q \geq 2_*$. Assume that $f \in L^{q}(U)$ and $F \in L^{2q}(U, \mathbb{R}^d)$. Then, the following hold:
\begin{itemize}
\item[(i)]
There exists $u \in H^{1,2}_0(U)_b$ such that $u$ is a weak solution to \eqref{undeq}. 
Moreover, \eqref{enerestim} and \eqref{linfinestim} are satisfied.

\item[(ii)]
There exists $w \in H^{1,2}_0(U)_b$ such that $w$ is a weak solution to \eqref{undeqdual}. Moreover, \eqref{enerestim} and \eqref{linfinestim} are satisfied where $u$ is replaced by $w$.
\end{itemize}
\end{theorem}
\begin{proof}
(i)  As in the proof of Theorem \ref{theo2.6}(iii), there exist $u \in H^{1,2}_0(U)$ and a sequence of functions $(u_n)_{n \geq 1}$ in $H^{1,2}_0(U)_b$ such that $\lim_{n \rightarrow \infty} u_n = u$ weakly in $H^{1,2}_0(U)$ and $\lim_{n \rightarrow \infty} u_n = u$ a.e. and that \eqref{(maineq)} and \eqref{enerestim} hold.  By Theorem \ref{boundthm}, \eqref{linfinestim} holds where $u$ is replaced by $u_n$ for each $n \geq 1$. Thus, $u \in H^{1,2}_0(U)_b$ and \eqref{linfinestim} follows.\\
(ii) Analogously to (i), the assertion for $w$ follows.
\end{proof}

\section{Proofs of the main results}\label{sec5}
\noindent The main idea for the proof of the uniqueness in 
Theorem \ref{theomain}(i) stems from \cite[Lemma 2.2.11]{K07} and \cite[Theorem 4.7(i)]{KT20} where the existence of a bounded weak solution to a dual problem is crucially used.

\centerline{}
\noindent
{\bf Proof of Theorem \ref{theomain}.}\\ 
(i)
Existence of a weak solution $u$ to \eqref{undeq} satisfying \eqref{enerestim} follows from Theorem \ref{theo2.6}(iii). And the existence of a bounded weak solution $u$ to \eqref{undeq} satisfying \eqref{linfinestim} follows from Theorem \ref{theo3.2}(i). In order to show the uniqueness of weak solutions to \eqref{undeq}, by linearity it is enough to show  the following claim:\\[5pt]
{\bf Claim}: Assume {\bf (A)}. Let $u \in H^{1,2}_0(U)$ satisfy that for any $\varphi \in C_0^{\infty}(U)$
\begin{equation}  \label{4.1}
\int_{U} \langle A \nabla u, \nabla \varphi \rangle   + \big( \langle B, \nabla u \rangle  + (c+\alpha)u\big) \varphi dx = 0.
\end{equation}
Then, $u=0$. \\ \\
To show the claim, let $\psi \in C_0^{\infty}(U)$ be arbitrarily fixed. By Theorem \ref{theo3.2}(ii) and using an approximation of $H^{1,2}_0(U)$ by $C_0^{\infty}(U)$, there exists $w \in H^{1,2}_0(U)_b$ such that for any $\varphi \in H^{1,2}_0(U)$
\begin{equation} \label{4.3}
\int_{U} \langle A^T \nabla w+ w B, \nabla \varphi \rangle  +  (c+\alpha) w \varphi dx = \int_{U} \psi \varphi dx.
\end{equation}
Replacing $\varphi$ by $u$ in \eqref{4.3}, we have
\begin{equation} \label{4.4}
\int_{U} \langle A^T \nabla w+ w B, \nabla u \rangle + (c+\alpha) w u dx = \int_{U} \psi u dx.
\end{equation}
Note that using an approximation, \eqref{4.1} holds for any $\varphi \in H^{1,2}_0(U)_b$. Thus, replacing $\varphi$ by $w$
\begin{equation} \label{4.5}
\int_{U} \langle A \nabla u, \nabla w \rangle+ \big(\langle B, \nabla u \rangle  + (c+\alpha)u \big) w dx = 0.
\end{equation}
Subtracting \eqref{4.5} from \eqref{4.4}, we have $\int_{U} \psi u dx =0$, and hence $u=0$. \\
(ii) By Theorem \ref{theo2.6}(v), $G_{\alpha} f \in  H^{1,2}_0(U) \cap  L^r(U)$ such that
\begin{equation*}
\int_{U} \langle A \nabla G_{\alpha}f, \nabla \varphi \rangle+  \langle B, \nabla  G_{\alpha} f \rangle \varphi + (c+\alpha) G_{\alpha}f  \cdot \varphi \,dx =  \int_{U}  f \varphi \, dx, \quad \text{for all $\varphi \in C_0^{\infty}(U)$}.
\end{equation*}
and that $\| G_{\alpha} f \|_{L^r(U)} \leq \alpha^{-1} \|f\|_{L^r(U)}$. Let $\hat{w}=u-G_{\alpha} f$. Then, by linearity $\hat{w} \in H^{1,2}_0(U)$ and it holds
\begin{equation*} \label{(fzerote)}
\int_{U} \langle A \nabla \hat{w}, \nabla \varphi \rangle+ \big( \langle B,  \nabla \hat{w} \rangle +(c+\alpha) \hat{w} \big)\varphi dx   = \int_{U} \langle F, \nabla \varphi \rangle dx, \quad \text{for all $\varphi \in C_0^{\infty}(U)$}. 
\end{equation*}
Thus, by Theorem \ref{theomain}(i), $\hat{w} \in L^{\infty}(U)$ and
there exists the constant $C_2>0$ as in (i) such that
\begin{equation*} \label{infestimze}
\|\hat{w}\|_{L^{\infty}(U)} \leq C_2 \| F\|_{L^{2q}(U)}. 
\end{equation*}
Therefore, $u = G_{\alpha} f +\hat{w} \in L^r(U)$ and
$$
\|u\|_{L^r(U)} \leq \|G_{\alpha} f\|_{L^r(U)} +\|\hat{w}\|_{L^r(U)} \leq \alpha^{-1}\|f\|_{L^r(U)} +|U|^{1/r} C_2 \| F\|_{L^{2q}(U)},
$$
as desired. \qed
\centerline{}
\centerline{}
In order to derive our $L^1$-stability result, the following extended version of the $L^1$-contraction property is needed. Since the integrability of $g$ and $G$ in the following theorem is lower than the one of $f$ and $F$ in Theorem \ref{theomain}, we will show below through a suitable approximation and a duality argument.

\begin{theorem} \label{theo4.2}
Assume {\bf (A)}.
Let $\alpha \in (0, \infty)$, $g \in L^{1}(U)$, $G \in L^2(U, \mathbb{R}^d)$ and $v \in H^{1,2}_0(U)$. Assume that $v$ is a weak solution to \eqref{undeq} where $f$ and $F$ are replaced by $g$ and $G$, respectively. Then, 
\begin{align} \label{l1estim}
\| v \|_{L^1(U)} \leq \alpha^{-1} \| g\|_{L^1(U)} +|U|^{1/2}C_1\|G\|_{L^2(U)},
\end{align}
where $C_1>0$ is the constant as in Theorem \ref{theo2.6}(i).
\end{theorem}

\begin{proof}
Let $q>d/2$ with $q \geq 2_*$, $\hat{g} \in L^{q}(U)$ and $\hat{G} \in L^{2q}(U, \mathbb{R}^d)$. \\
We first claim that there exists $\hat{v} \in H^{1,2}_0(U)_b$ such that $\hat{v}$ is a weak solution to \eqref{undeq} where $f$ and $F$ are replaced by $\hat{g}$  and $\hat{G}$, respectively and that
\begin{align*} 
\| \hat{v} \|_{L^1(U)} \leq \alpha^{-1} \| \hat{g}\|_{L^1(U)} +|U|^{1/2}C_1\|\hat{G}\|_{L^2(U)}.
\end{align*}
Indeed, the existence and uniqueness of weak solution $\hat{v} \in H^{1,2}_0(U)_b$ follow by Theorem \ref{theomain}(i). By Theorems 
\ref{theomain} and \ref{theo2.6}(v),  we obtain that $G_{\alpha} \hat{g} \in H^{1,2}_0(U)_b$ and  $G_{\alpha} \hat{g}$ is a unique weak solution to \eqref{undeq} where $f$ and $F$ are replaced by $\hat{g}$ and $0$, and it holds that
$$
\|G_{\alpha} \hat{g}\|_{L^1(U)}   \leq \alpha^{-1} \| \hat{g} \|_{L^1(U)}.
$$  
Let $\hat{w} = \hat{v}- G_{\alpha} \hat{g}$. Then, by linearity $\hat{w} \in H^{1,2}_0(U)_b$ and $\hat{w}$ is a unique weak solution to \eqref{undeq} where $f$ and $F$ are replaced by $0$ and $\hat{G}$, respectively. Hence it follows from \eqref{enerestim} in Theorem \ref{theo2.6}(i) that  
$$
\|\hat{w}\|_{L^1(U)} \leq |U|^{1/2} \|\hat{w}\|_{H_0^{1,2}(U)} \leq |U|^{1/2}C_1\|\hat{G}\|_{L^2(U)},
$$
where $C_1>0$ is the constant as in \eqref{enerestim}. Since 
$$
\|\hat{v}\|_{L^1(U)}   \leq \|G_{\alpha} \hat{g} \|_{L^1(U)} +\|\hat{w}\|_{L^1(U)}
\leq \alpha^{-1} \| \hat{g}\|_{L^1(U)} +|U|^{1/2} C_1 \|\hat{G}\|_{L^2(U)},
$$
 the claim follows. \\
Now let $(g_n)_{n \geq 1}$ and $(G_n)_{n \geq 1}$ be sequences of functions and vector fields in $L^{\infty}(U)$ and $L^{\infty}(U, \mathbb{R}^d)$, respectively, such that $\lim_{n \rightarrow \infty} g_n = g$ in $L^1(U)$ and $\lim_{n \rightarrow \infty}G_n =G$ in $L^2(U, \mathbb{R}^d)$. 
By Theorem \ref{theomain}(i), for each $n \geq 1$ there exists $\widetilde{v}_n \in H^{1,2}_0(U)_b$ such that $\widetilde{v}_n$ is a unique weak solution to \eqref{undeq}, where $f$ and $F$ are replaced by $g_n$ and $G_n$, respectively. 
Moreover, the claim yields that for each $n, m \geq 1$
\begin{equation} \label{complel1}
\|\widetilde{v}_n -\widetilde{v}_m \|_{L^1(U)} \leq \alpha^{-1} \|g_n-g_m \|_{L^1(U)}  + |U|^{1/2}C_1\| G_n- G_m \|_{L^2(U)}
\end{equation}
and that
\begin{align*} 
\| \widetilde{v}_n \|_{L^1(U)} \leq \alpha^{-1} \|g_n\|_{L^1(U)} +|U|^{1/2}C_1\|G_n\|_{L^2(U)}.
\end{align*}
Using \eqref{complel1} and the completeness of $L^1(U)$, there exists $\widetilde{v} \in L^1(U)$ such that $\lim_{n \rightarrow \infty} \widetilde{v}_n = \widetilde{v}$ in $L^1(U)$, so that \eqref{l1estim} holds where $v$ is replaced by $\widetilde{v}$. To complete our assertion, we will show that $v = \widetilde{v}$ by using a duality argument. Let $\psi \in C_0^{\infty}(U)$ be arbitrarily fixed. By Theorem \ref{theo3.2}(ii), there exists $w \in H^{1,2}_0(U)_b$ such that 
\begin{align} \label{firstvarid}
&\int_{U} \langle A^T \nabla w+ w B, \nabla (v-\widetilde{v}_n) \rangle  +(c+\alpha) w (v-\widetilde{v}_n) dx = \int_{U} \psi (v-\widetilde{v}_n) dx.
\end{align}
Since $v-\widetilde{v}_n$ is a weak solution to \eqref{undeq} where $f$ and $F$ are replaced by $g-g_n$ and $G-G_n$, respectively,  we obtain that
\begin{align}
&\int_{U} \langle A \nabla (v-\widetilde{v}_n), \nabla w \rangle + \big(\langle B, \nabla (v-\widetilde{v}_n) \rangle + (c+\alpha) (v-\widetilde{v}_n) \big)w dx  \nonumber \\
&\quad = \int_{U} (g-g_n) w+ \langle G-G_n, \nabla w\rangle  dx. \label{secondvarid}
\end{align}
Subtracting \eqref{secondvarid} from \eqref{firstvarid}, 
$
\int_{U} \psi (v-\widetilde{v}_n) dx = \int_{U} (g-g_n) w + \langle (G-G_n), \nabla w\rangle  dx.
$
Passing to the limit $n \rightarrow \infty$, we get $\int_{U} \psi(v-\widetilde{v}) dx = 0$, and hence $v=\widetilde{v}$ as desired.
\end{proof}
\text{} \\
{\bf Proof of Corollary \ref{stabil}.}\;\\
Note that for each $n \geq 1$ and $\varphi \in C_0^{\infty}(U)$
\begin{align*}
&\int_{U} \langle A_n \nabla (u-u_n), \nabla \varphi \rangle +\big( \langle B, \nabla (u-u_n) \rangle+ (c+\alpha) (u-u_n) \big)\, \varphi dx \\
& =  \int_{U} \big( \langle B_n -B, \nabla u_n   \rangle + (c_n-c)u_n +f-f_n\big) \varphi + \langle (A_n-A) \nabla u + F-F_n, \nabla \varphi \rangle dx.
\end{align*}
Using Theorem \ref{theo4.2}, 
we obtain that
\begin{align*}
& \|u_n-u\|_{L^1(U)} \leq \alpha^{-1} \big\| \langle B -B_n, \nabla u_n   \rangle + (c-c_n)u_n   \big\|_{L^1(U)} \\
& \qquad + C_3 \|(A_n-A) \nabla u \|_{L^2(U)} + \alpha^{-1} \|f-f_n\|_{L^1(U)} + C_3 \|F-F_n\|_{L^2(U)},
\end{align*}
where $C_3=|U|^{1/2} C_1$ and $C_1>0$ is the constant as in Theorem \ref{theomain}(i). Hence, \eqref{stabilest} is established by H\"{o}lder's inequality, Sobolev's inequality and Theorem \ref{theo2.6}(i). The rest follows from Lebesgue's theorem. \qed

\centerline{}

\noindent
{\bf Acknowledgment.} The author would like to express his sincere gratitude to the anonymous referee for giving valuable comments and suggestions to improve the paper.

\centerline{}

\bibliographystyle{amsplain}

\end{document}